\documentclass[11pt]{amsart}

\usepackage{amsfonts}
\usepackage{amsthm}
\usepackage{amssymb}
\usepackage{amsmath}
\usepackage{url, hyperref}

\newtheorem{theo}{\bf Theorem}[section]
\newtheorem{lemma}{\bf Lemma}[section]
\newtheorem{coro}{\bf Corollary}[section]

\newcommand{\C}{{\mathcal C}}

\newcommand{\B}{{\mathcal B}}

\newcommand{\Z}{{\mathbb Z}}

\newcommand{\Q}{{\Bbb Q}}
\newcommand{\R}{{\mathbb R}}

\newcommand{\Sp}{{\mathcal S}}

\newcommand{\bea}{\begin{eqnarray*}}
\newcommand{\eea}{\end{eqnarray*}}
\newcommand{\be}{\begin{eqnarray}}
\newcommand{\ee}{\end{eqnarray}}

\newcommand{\ve}{\boldsymbol}

\newcommand{\rank}{\mbox{\rm rank}\,}

\newcommand{\frob}{{F}}

\newcommand{\brauer}{{G}}

\newcommand{\modulo}{\,\mathrm{mod}\;}

\numberwithin{equation}{section}

\begin{document}

\title[Linear Diophantine problems]{On polynomial-time solvable linear Diophantine problems}
\author{Iskander Aliev}
\address{Mathematics Institute, Cardiff University, Cardiff, Wales, UK}
\email{alievi@cardiff.ac.uk}



\begin{abstract}
We obtain a polynomial-time algorithm that, given input $(A, {\ve b})$, where $A=(B|N)\in \Z^{m\times n}$, $m<n$, with nonsingular $B\in\Z^{m\times m}$  and ${\ve b}\in\Z^m$, finds a nonnegative integer solution to the system  $A{\ve x}={\ve b}$ or determines that no such  solution exists, provided that ${\ve b}$ is located sufficiently ``deep'' in the cone generated by the columns of $B$. This result improves on some of the previously known conditions that guarantee polynomial-time solvability of linear Diophantine problems.

\end{abstract}



\maketitle

\section{Introduction and Statement of Results}

Consider the linear Diophantine problem
\be
\begin{array}{l}
\mbox{Given } (A, {\ve b})\,, \mbox{ where }\; A\in\Z^{m\times n}, m<n, \rank(A)=m \mbox{ and }\\ {\ve b}\in \Z^m,  \mbox{ find a nonnegative integer solution  to  the system }\\ A{\ve x}= {\ve b}\;\mbox{or determine that no such  solution exists}\,.
\end{array}
\label{linear_equation}
\ee
The problem (\ref{linear_equation}) is referred to as the {\em  multidimensional knapsack problem} and is NP-hard already for $m=1$ (see Papadimitriou and Steiglitz \cite[Section 15.7]{PS}).

Let ${\ve v}_1,\ldots,{\ve v}_n\in \Z^m$  be the columns  of the matrix $A$ and let
\bea
\C_A=\{\lambda_1{\ve v}_1+\cdots+\lambda_n{\ve v}_n: \lambda_1,\ldots,\lambda_n\ge 0\}\,
\eea
be the cone generated by ${\ve v}_1,\ldots,{\ve v}_n$. In this paper, we are interested in the problem of determining subsets $\Sp \subset \C_A$
such that  (\ref{linear_equation}) is solvable in polynomial time provided  ${\ve b}\in \Sp$.
We will use the general approach of Gomory  \cite{Gomory_polyhedra}, that was originally applied to study asymptotic integer programs, and combine it with results from discrete geometry.

We may assume, without loss of generality, that the matrix $A$ is partitioned as
\bea
A=(B|N)\,,
\eea
where $B\in \Z^{m\times m}$ is nonsingular and $N\in \Z^{m\times (n-m)}$.
In what follows, we will denote by  $l_B$ and $l_N$ the Euclidean lengths of the longest columns in the matrices $B$
and $N$, respectively.

Let $\C_B\subset \C_A$ be the cone generated by the columns of the matrix $B$. The main result of this paper shows that (\ref{linear_equation}) is solvable in polynomial time when the right-hand-side vector ${\ve b}$ is located deep enough in the cone $\C_B$.

Let $\C_B(t)\subset \C_B$ denote the affine cone of points in $\C_B$ at Euclidean distance  $\ge t$ from the boundary of $\C_B$. We will denote by $\gcd(A)$  the greatest common divisor
of all  $m\times m$ subdeterminants of $A$.

\begin{theo}\label{PolySubCone2}
There exists a polynomial-time algorithm which, given input $(A, {\ve b})$, where $A=(B | N)\in \Z^{m\times n}$, with nonsingular $B\in\Z^{m\times m}$, and
\be {\ve b}\in \Z^m\cap\C_B\left(l_N\left(\frac{|\det(B)|}{\gcd(A)}-1\right)\right)\,,\label{Gomory_cone}\ee
finds a nonnegative integer solution to the system $A{\ve x}= {\ve b}$ or determines that no such  solution exists.
%
\end{theo}

We will now consider a special case where the matrix $A$ satisfies the following conditions:
\begin{equation}
\begin{split}
{\rm (i)}&\,\, \gcd(A)=1, \\
{\rm (ii)}&\,\, \{{\ve x}\in\R^n_{\ge 0}: A\,{\ve x}={\ve 0}\}=\{{\ve 0}\}.
\end{split}
\label{positive_primitive_a}
\end{equation}
%

Notice that the condition (i) in  (\ref{positive_primitive_a}) guarantees that the system $A{\ve x}={\ve b}$ has an integer solution for each ${\ve b}\in \Z^m$ (see Schrijver \cite[Corollary 4.1 c]{Schrijver}).
The condition (ii) in  (\ref{positive_primitive_a}), in its turn, guarantees that the polyhedron
$\{{\ve x}\in \R^{n}_{\ge 0}: A{\ve x}={\ve b}\}$
is bounded.

When $m=1$ in the setting  (\ref{positive_primitive_a}), the problem (\ref{linear_equation}) is linked to the well-known {\em Frobenius problem} (see Ramirez Alfonsin \cite{Alf}).
By the condition  (i) in  (\ref{positive_primitive_a}), we have $\gcd(a_{11}, \ldots, a_{1n})=1$ and by (ii) we may assume that the entries of $A$ are positive. For such $A$ the largest integer $b$ such that (\ref{linear_equation}) is infeasible is called the {\em Frobenius number} associated with $A$, denoted by $\frob(A)$.
It is an interesting question to determine whether there exists a polynomial-time algorithm that solves  (\ref{linear_equation}) provided that
\bea
b>\frob(A)\,
\eea
(cf. Conjecture 1.1 in \cite{LLL-knapsacks}).

The best known result in this direction is due to Brimkov \cite{Brimkov_best} (see also \cite{LLL-knapsacks}, \cite{Brimkov} and \cite{Brimkov2}). Specifically, set
\be\label{corners_f}
f_1=a_{11}, f_i=\gcd(a_{11},\ldots, a_{1i})\,,\; i\in \{2, \ldots, n\}\,.
\ee
A classical upper bound of Brauer \cite{Brauer} for the Frobenius numbers states that
\be\label{Brauer_bound}
\frob(A)\le \brauer(A):=a_{12}\frac{f_1}{f_2}+\cdots+a_{1 n}\frac{f_{n-1}}{f_{n}}-\sum_{i=1}^{n} a_{1i}\,.
\ee
Brauer \cite{Brauer} and, subsequently, Brauer and Seelbinder \cite{BrauerSeelbinder} proved that the bound (\ref{Brauer_bound}) is sharp and obtained a necessary and sufficient condition for the equality $\frob(A)= \brauer(A)$.
Brimkov \cite{Brimkov_best}  gave a polynomial-time algorithm that solves  (\ref{linear_equation}) provided that
\be\label{Brimkov_bound}
b>\brauer(A)\,.
\ee
We will show that an algorithm obtained in the proof of Theorem \ref{PolySubCone2} matches the bound (\ref{Brimkov_bound}).
\begin{coro}\label{PT2}
There exists a polynomial-time algorithm which, given input $(A, {b})$, where $A\in\Z^{1\times n}_{>0}$ satisfies (\ref{positive_primitive_a}) and $b\in \Z$ satisfies
\bea b> \brauer(A) \,, \eea
computes a nonnegative integer solution to the equation $A{\ve x}={b}$.
%
\end{coro}

Recall that the {\em Minkowski sum} $X+Y$ of the sets $X, Y\subset \R^m$ consists of all points ${\ve x}+{\ve y}$ with ${\ve x}\in X$ and ${\ve y}\in Y$.
For $m\ge 2$, Aliev and Henk \cite{LLL-knapsacks} considered the problem of estimating the minimal $t=t(A)\ge 0$ such that the problem  (\ref{linear_equation})
is solvable in polynomial time provided that $A$ satisfies (\ref{positive_primitive_a}) and
\bea{\ve b}\in \Z^m\cap(t{\ve v}+\C_A)\,,\eea
where ${\ve v}={\ve v}_1+\cdots+{\ve v}_n$ is the sum of columns of $A$.

Theorem 1.1 in \cite{LLL-knapsacks} gives the bound
\be t\le 2^{(n-m)/2-1}p(m,n)(\det(AA^T))^{1/2}\,,\label{desired_cone}\ee
where
\bea p(m,n)=2^{-1/2}(n-m)^{1/2}n^{1/2}\,.\eea
Furthermore, Theorem 1.2 in \cite{LLL-knapsacks} shows that the exponential factor $2^{(n-m)/2-1}$ in (\ref{desired_cone}) is redundant for matrices with
\be
\det(AA^T)> \frac{(n-m)2^{2(n-m-2)}\gamma_{n-m}^{n-m}}{n^2}\,.
\label{lower_det}
\ee
Here $\gamma_k$ is the $k$-dimensional Hermite constant for which we
refer to \cite[Definition 2.2.5]{Martinet}.

Let us now consider the case $m=2$. Condition (\ref{positive_primitive_a})  (ii) implies that the cone $\C_A$ is pointed. Thus we may assume without loss of generality that $A=(B|N)$
with $\C_B=\C_A$. The last result of this paper gives an estimate on the function $t(A)$ that is independent on the dimension $n$
and allows a refinement of (\ref{desired_cone})
when the ratio $l_Bl_N/|\det(B)|$ is relatively small.

\begin{coro}
There exists a polynomial-time algorithm which, given input $(A, {\ve b})$, where $A=(B | N) \in \Z^{2\times n}$,   $B\in\Z^{2\times 2}$ is nonsingular with $\C_B=\C_A$, $A$ satisfies (\ref{positive_primitive_a}) and
\be{\ve b}\in \Z^2\cap\left(\frac{l_Bl_N}{|\det(B)|}\left(|\det(B)|-1\right){\ve v}+\C_A\right )\,,\label{desired_cone_1}\ee
computes a nonnegative integer solution to the system $A{\ve x}={\ve b}$.
\label{polynomial_algorithm_coro}
\end{coro}
Noticing that $|\det(B)|\le (\det(AA^T))^{1/2}$, the condition (\ref{desired_cone_1}) improves on (\ref{desired_cone}) provided that
$l_Bl_N/|\det(B)|\le 2^{(n-m)/2-1}p(m,n)$. For matrices $A$ satisfying (\ref{lower_det}) an improvement occurs when $l_Bl_N/|\det(B)|\le p(m,n)$.

\section{Tools from discrete geometry}

For linearly independent ${\ve b}_1, \ldots, {\ve b}_k$ in $\R^d$, the set $\Lambda=\{\sum_{i=1}^{k} \lambda_i {\ve b}_i:\, \lambda_i\in \Z\}$ is a $k$-dimensional {\em lattice} with {\em basis} ${\ve b}_1, \ldots, {\ve b}_k$ and {\em determinant}
$\det(\Lambda)=(\det({\ve b}_i\cdot {\ve b}_j)_{1\le i,j\le k})^{1/2}$, where ${\ve b}_i\cdot {\ve b}_j$ is the standard inner product of the basis vectors ${\ve b}_i$ and  ${\ve b}_j$.
For a lattice $\Lambda\subset\R^d$ and ${\ve y}\in \R^d$, the set ${\ve y}+\Lambda$ is an {\em affine lattice} with determinant $\det(\Lambda)$.

Let $\Lambda$ be a lattice in $\R^d$ with basis ${\ve b}_1, \ldots, {\ve b}_d$ and
let $\hat{\ve b}_i$ be the vectors obtained from the Gram-Schmidt orthogonalisation of ${\ve b}_1, \ldots, {\ve b}_d$:
\be
\label{GS}
\begin{array}{l}
\hat{\ve b}_1= {\ve b}_1\,, \\
\hat{\ve b}_{i}= {\ve b}_i-\sum_{j=1}^{i-1}\mu_{i,j}\hat{\ve b}_j\,, \;\;j\in\{2,\ldots,d\}\,,
\end{array}
\ee
where $\mu_{i,j}=({\ve b}_i\cdot\hat{\ve b}_j)/|\hat{\ve b}_j|^2$.

We will associate with the basis ${\ve b}_1, \ldots, {\ve b}_d$ of $\Lambda$ the
box
\bea
{\hat \B}({\ve b}_1, \ldots, {\ve b}_d)=[0,\hat{\ve b}_1)\times [0,\hat{\ve b}_2) \times \cdots \times [0, \hat{\ve b}_d)\,.
\eea
\begin{lemma}\label{point_in_GS_box}
There exists a polynomial-time algorithm that, given a basis ${\ve b}_1, \ldots, {\ve b}_d$ of a $d$-dimensional lattice $\Lambda\subset \Q^d$ and a point ${\ve x}$ in $\Q^d$ finds a point ${\ve y}\in \Lambda$ such that ${\ve x}\in {\ve y}+{\hat \B}({\ve b}_1, \ldots, {\ve b}_d)$.
\end{lemma}

A proof of Lemma \ref{point_in_GS_box} is implicitly contained, for instance, in the description of the classical nearest plane procedure of Babai \cite{Babaika}. For completeness, we include a proof that follows along an argument of the proof of Theorem 5.3.26 in \cite{GLS}.

\begin{proof}

Let ${\ve x}$ be any point of $\Q^d$. We need to find a point ${\ve y}\in \Lambda$ such that
\be\label{GS-box-representation}
{\ve x}-{\ve y}=\sum_{i=1}^d \lambda_i \hat{\ve b}_i\,,\;\; \lambda_i\in [0, 1)\,, \;\;i \in \{1, \ldots,  d\}\,.
\ee
This can be achieved using the following procedure. First, we find the rational numbers $\lambda_i^0$, $i \in \{1, \ldots,  d\}$
such that
\bea
{\ve x}=\sum_{i=1}^d \lambda_i^0 \hat{\ve b}_i\,.
\eea
This can be done in polynomial time by Theorem 3.3 in \cite{Schrijver}.
Then we subtract $ \lfloor\lambda_d^0 \rfloor {\ve b}_d$ to get a representation
\bea
{\ve x}-\lfloor\lambda_d^0 \rfloor {\ve b}_d=\sum_{i=1}^d \lambda_i^1 \hat{\ve b}_i\,,
\eea
where $\lambda_d^1\in [0, 1)$. Next subtract $ \lfloor\lambda_{d-1}^1 \rfloor {\ve b}_{d-1}$ and so on until we obtain the representation
(\ref{GS-box-representation}). 
\end{proof}

Let now $\Lambda$ be a $d$-dimensional sublattice of $\Z^{d}$.
By Theorem I (A) and Corollary 1 in Chapter I of Cassels \cite{Cassels}, there exists a unique basis ${\ve g}_1, \ldots, {\ve g}_d$ of the sublattice $\Lambda$
of the form
\be\label{special_basis}
\begin{array}{l}
{\ve g}_1= v_{11}{\ve e}_1\,,\\
{\ve g}_2= v_{21}{\ve e}_1+v_{22}{\ve e}_2\,,\\
\vdots\\
{\ve g}_d= v_{d1}{\ve e}_1+\cdots+v_{dd}{\ve e}_d\,,
\end{array}
\ee
where ${\ve e}_i$ are the standard basis vectors of $\Z^d$ and the coefficients $v_{ij}$ satisfy the conditions
$v_{ij}\in \Z, v_{ii}>0 \mbox{ for }i\in\{1,\ldots,d\} $ and $0\le v_{ij}<v_{jj} \mbox{ for }i,j\in\{1,\ldots,d\},i>j$.

\begin{lemma}\label{comp_special_basis}
There exists a polynomial-time algorithm that, given a basis ${\ve b}_1, \ldots, {\ve b}_d$ of a lattice $\Lambda\subset \Z^d$  finds the basis of $\Lambda$ of the form (\ref{special_basis}).
\end{lemma}

\begin{proof}

Let  $V=(v_{ij})\in \Z^{d\times d}$ be the matrix formed by the coefficients $v_{ij}$ in (\ref{special_basis})
with $v_{ij}=0$ for $j>i$. Observe that after a straightforward re-numbering of the rows
and columns of $V$ we obtain a matrix in the row-style Hermite Normal Form.
Now it is  sufficient to notice that the Hermite Normal Form can be computed in polynomial time using an algorithm of Kannan and Bachem \cite{Kannan_Bachem}.
\end{proof}

The Gram-Schmidt orthogonalisation (\ref{GS}) of the basis (\ref{special_basis}) of $\Lambda$ has the form $\hat {\ve g}_1= v_{11}{\ve e}_1, \ldots,  \hat{\ve g}_d=v_{dd}{\ve e}_d$. Therefore, noticing that the basis (\ref{special_basis}) is unique, we can associate with $\Lambda$ the box
\bea
\begin{array}{l}
\B(\Lambda)=\hat \B({\ve g}_1, \ldots, {\ve g}_d)=[0,v_{11})\times [0,v_{22})\times  \cdots \times [0, v_{dd})\,.
\end{array}
\eea

\begin{lemma}\label{prod_det}
For any ${\ve w}=(w_1, \ldots, w_d)^T\in \B(\Lambda) \cap \Z^d$ we have
\bea
\prod_{i=1}^d (1+ w_i) \le \det(\Lambda)\,.
\eea
\end{lemma}

\begin{proof}
It is sufficient to notice that by (\ref{special_basis}) $ \det(\Lambda)=v_{11}\cdots v_{dd}$.
\end{proof}

\section{Proof of Theorem \ref{PolySubCone2}}

Given $A\in\Z^{m\times n}$ and ${\ve b}\in\Z^m$, we will denote by $\Gamma(A,{\ve b})$  the set of integer points in the affine subspace
\bea
\Sp({A,{\ve b}})= \{{\ve x}\in \R^n: A{\ve x}={\ve b}\}\,,
\eea
 that is
\bea \Gamma(A,{\ve b})=\Sp({A,{\ve b}})\cap \Z^n\,.\eea
The set $\Gamma(A,{\ve b})$ is either empty or is an affine lattice of the form $\Gamma(A,{\ve b})= {\ve r}+ \Gamma(A)$,
where ${\ve r}$ is any integer vector with $A{\ve r}={\ve b}$  and $\Gamma(A) = \Gamma(A,{\ve 0})$ is the  lattice formed by all integer points in the kernel of the matrix $A$.
We will call the system  $A{\ve x}={\ve b}$ {\em integer feasible} if it has integer solutions or, equivalently, $\Gamma(A,{\ve b})\neq \emptyset$. Otherwise the system is called {\em integer infeasible}.

Let $\pi$ denote the projection map from $\R^n$ to $\R^{n-m}$ that forgets the first $m$ coordinates. Recall that Theorem \ref{PolySubCone2} applies to $A=(B|N)$, where $B$ is nonsingular. It follows that  the restricted map $\pi |_{\Sp({A,{\ve b}})}:\Sp({A,{\ve b}})\rightarrow \R^{n-m} $ is bijective. Specifically, for any ${\ve w}\in \R^{n-m}$ we have
\bea
\pi |_{\Sp({A,{\ve b}})}^{-1}({\ve w})= \left(
\begin{array}{c}{\ve u}\\ {\ve w}\end{array}
\right)\,\mbox{ with } {\ve u}= B^{-1}({\ve b}-N{\ve w})\,.
\eea

For technical reasons, it is convenient to consider the projected set $\Lambda(A,{\ve b})=\pi(\Gamma(A,{\ve b}))$ and the projected lattice $\Lambda(A)=\pi(\Gamma(A))$.
Since the map $\pi |_{\Sp({A,{\ve 0}})}$ is bijective, we  obtain the following lemma.

\begin{lemma}\label{projected_basis} Let ${\ve g}_1, \ldots, {\ve g}_{n-m}$ be a basis of $\Gamma(A)$.
The vectors ${\ve b}_1=\pi({\ve g}_1), \ldots, {\ve b}_{n-m}=\pi({\ve g}_{n-m})$ form a basis of the lattice $\Lambda(A)$.
\end{lemma}
Using notation of Lemma \ref{projected_basis},  let $G\in \Z^{n \times (n-m)}$ be the matrix with columns  ${\ve g}_1, \ldots, {\ve g}_{n-m}$.
We will denote by $F$ the $(n-m)\times (n-m)$-submatrix of $G$ consisting of the last $n-m$
rows; hence, the columns of $F$ are ${\ve b}_1, \ldots, {\ve b}_{n-m}$. Then $\det(\Lambda(A))=|\det(F)|$.
The rows of the matrix $A$ span the $m$-dimensional rational subspace of $\R^n$ orthogonal to the $(n-m)$-dimensional rational subspace spanned by the columns of $G$. Therefore, by Lemma 5G and Corollary 5I in \cite{Schmidt_Approximations}, we have $|\det(F)|=|\det(B)|/\gcd(A)$ and, consequently,

%
\be\label{determinant_via_submatrix}\det(\Lambda(A))=\frac{|\det(B)|}{\gcd(A)}\,.\ee
%


Consider the following algorithm.

\vskip.5cm
\noindent{\bf Algorithm 1}
\vskip.3cm
\begin{itemize}

\item[{\em Input:}]  $(A, {\ve b})$, where $A=(B|N)\in \Z^{m\times n}$, $m<n$, with nonsingular $B\in \Z^{m\times m}$ and ${\ve b}\in \Z^m$.

\item[{\em Output:}] Solution ${\ve x}\in \Z^n$ to an integer feasible system $A{\ve x}={\ve b}$.
\vskip.2cm

\item[{\em Step 0:}] If $\Gamma(A, {\ve b})=\emptyset$ then the system $A{\ve x}={\ve b}$ is integer infeasible. Stop.

\item[{\em Step 1:}] Compute a point ${\ve z}$ of the affine lattice $\Lambda(A, {\ve b})$.

\item[{\em Step 2:}] Find a point ${\ve y}\in \Lambda(A)$ such that ${\ve z}\in {\ve y}+{\B}(\Lambda(A))$.

\item[{\em Step 3:}] Set ${\ve w}={\ve z}-{\ve y}$ and output the vector 
\be\label{lifting}
{\ve x}=
\left(
\begin{array}{c}{\ve u}\\ {\ve w}\end{array}
\right) \mbox{ with }{\ve u}= B^{-1}({\ve b}-N{\ve w})\,.
\ee

\end{itemize}
\vskip.5cm

Note that Algorithm 1 will be also used in the proof of Corollary \ref{PT2}, where the condition (\ref{Gomory_cone}) is replaced by its refinement (\ref{Brimkov_bound}). For this reason,  we do not require that the input of the algorithm satisfies (\ref{Gomory_cone}) and, as a consequence, the algorithm outputs a certain integer, but not necessarily nonnegative
solution to an integer feasible system $A{\ve x}={\ve b}$ or detects integer infeasibility.

To complete the proof of Theorem \ref{PolySubCone2}, it is sufficient to show that Algorithm 1 is polynomial-time and that this algorithm computes a nonnegative integer solution to any integer feasible system $A{\ve x}={\ve b}$ that satisfies its input conditions together with (\ref{Gomory_cone}).


Let us show that all steps of the Algorithm 1 can be computed in polynomial time.
By Corollaries 5.3 b,c in \cite{Schrijver} we can compute in polynomial time integer vectors ${\ve r}, {\ve g}_1, \ldots, {\ve g}_{n-m}$ such that
\be\label{repres}\begin{split}
\Gamma(A,{\ve b})= {\ve r}+ \sum_{i=1}^{n-m}\lambda_i {\ve g}_i\,,
 \lambda_i\in \Z\,, i\in\{1, \ldots, n-m\}
 \end{split}
\ee
or determine that $\Gamma(A,{\ve b})$ is empty.  This settles Step 0 and Step 1.
%
Further,  the vectors ${\ve g}_1, \ldots, {\ve g}_{n-m}$ in (\ref{repres}) form a basis of  the lattice $\Gamma(A) $. In Step 2 we first find the projected vectors ${\ve b}_1=\pi({\ve g}_1), \ldots, {\ve b}_{n-m}=\pi({\ve g}_{n-m})$ that form a basis of the lattice $\Lambda(A)$ by Lemma \ref{projected_basis}. Then  the point  ${\ve y}$ can be computed in polynomial time using Lemmas \ref{comp_special_basis} and \ref{point_in_GS_box}.
Finally, the lifted point ${\ve x}$ in Step 3 is computed in polynomial time by a straightforward  calculation (\ref{lifting}).

We will now show that Algorithm 1 computes a nonnegative integer solution to any integer feasible system $A{\ve x}={\ve b}$ with $(A, {\ve b})$ satisfying its input conditions together with (\ref{Gomory_cone}). By Step 0, we may assume that $\Gamma(A, {\ve b})\neq \emptyset$ and hence at Step 1 we can find a point ${\ve z}\in \Lambda(A, {\ve b})$. At Step 2 we can find a point ${\ve y}\in \Lambda(A)$ with ${\ve z}\in {\ve y}+{\B}(\Lambda(A))$ by Lemma \ref{point_in_GS_box}. Hence, the point ${\ve w}={\ve z}-{\ve y}$ at Step 3 is a nonnegative point of the affine lattice  $\Lambda(A, {\ve b})$. 
%
Further, since ${\ve w}\in \Lambda(A, {\ve b})$ and $\pi |_{\Sp({A,{\ve b}})}$ is bijective, the point  ${\ve x}= \pi |_{\Sp({A,{\ve b}})}^{-1}({\ve w})$ is integer. Summarising, we have
\be\label{summary}
{\ve x}=
\left(
\begin{array}{c}{\ve u}\\ {\ve w}\end{array}
\right) \in {\Sp(A,{\ve b})}\cap\Z^n \mbox{ and }  \pi({\ve x})={\ve w}\ge {\ve 0}\,.
\ee
It is now sufficient to show that ${\ve u}\ge {\ve 0}$.

Observe that, by construction, ${\ve w}\in {\B}(\Lambda(A))$. Hence, Lemma \ref{prod_det}, applied to ${\ve w}$ and $\Lambda= \Lambda(A)$, implies
\be\label{prod_det_ineq}\begin{split}
\prod_{i=1}^{n-m} (1+ w_i) \le \det(\Lambda(A))\,.
\end{split}
\ee
Expanding the product in (\ref{prod_det_ineq}) gives
\bea\label{sum_det_ineq}\begin{split}
\sum_{i=1}^{n-m} w_i \le\det(\Lambda(A))-1\,.
\end{split}
\eea
Hence, denoting by $\|\cdot\|_2$ the Euclidean norm, we obtain the inequality
\be\label{Nw}
\|N{\ve w}\|_2 \le l_{N} \sum_{i=1}^{n-m} w_i \le l_{N} ( \det(\Lambda(A))-1)\,.
\ee
By (\ref{determinant_via_submatrix}), ${\ve b}\in \C_B(l_{N} (\det(\Lambda(A))-1))$
and by (\ref{Nw}), ${\ve b}-N{\ve w} \in \C_B$.
The cone $\C_B$  can be written as
\bea 
\C_B=\{{\ve y}\in\R^m: B^{-1}{\ve y} \ge {\ve 0}\}\,
\eea
and therefore
\bea
{\ve u}= B^{-1}({\ve b}-N{\ve w})\ge {\ve 0}\,.
\eea
\qed

\section{Proof of Corollary \ref{PT2}}\label{m_one}

Let $A=(a_{11}, \ldots, a_{1n})\in \Z^{1\times n}$ satisfy (\ref{positive_primitive_a}). Then the lattice $\Lambda(A)$ can be written in the form
\bea
\Lambda(A)=\{{\ve x}\in \Z^{n-1}: a_{12} x_1+\cdots +a_{1 n} x_{n-1} \equiv 0 \,(\modulo a_{11})\}\,.
\eea
Note also that $\det(\Lambda(A))=a_{11}$ by (\ref{determinant_via_submatrix}).

The next lemma shows that the  box $B(\Lambda(A))$ is entirely determined by the parameters $f_i$ defined by (\ref{corners_f}).
\begin{lemma}\label{boxshape}
The box $B=B(\Lambda(A))$ has the form
\bea
B=\left[0,\frac{f_1}{f_2}\right)\times \left [0,\frac{f_2}{f_3}\right)\times \cdots \times \left [0, \frac{f_{n-1}}{f_{n}}\right)\,.
\eea
\end{lemma}

\begin{proof}

By the definition of the box $B(\Lambda(A))$, it is sufficient to show that
\be\label{boxcorners}
v_{11}=\frac{f_1}{f_2}, v_{22}=\frac{f_2}{f_3}, \ldots, v_{n-1\,n-1}=\frac{f_{n-1}}{f_{n}}\,.
\ee
Let ${\ve g}_1, \ldots, {\ve g}_{n-1}$ be the basis of the form (\ref{special_basis}) of the lattice $\Lambda(A)$. Let $\Lambda_i(A)$ denote the sublattice of $\Lambda(A)$
generated by the first $i$ basis vectors ${\ve g}_1, \ldots, {\ve g}_i$.
We can write $\Lambda_i(A)$ in the form
\bea
\begin{split}
\Lambda_i(A)=\left\{(x_1, \ldots, x_i, 0, \ldots, 0)^T\in \Z^{n-1}:\right.
\frac{a_{12}}{f_{i+1}} x_1+\cdots +\frac{a_{1i+1}}{f_{i+1}} x_i \equiv 0\,\\
\left.\left(\modulo \frac{a_{11}}{f_{i+1}}\right)\right\}\,.
\end{split}
\eea
Hence, $\det(\Lambda_i(A))=a_{11}/f_{i+1}$, $i\in\{1,\ldots, n-1\}$.
On the other hand,  (\ref{special_basis}) implies $\det(\Lambda_i(A))= v_{11}v_{22}\cdots v_{ii}$, $i\in\{1,\ldots, n-1\}$. Since $a_{11}=\det(\Lambda(A))=v_{11}v_{22}\cdots v_{n-1\, n-1}$, we have
$f_{i+1}=v_{i+1\, i+1}\cdots v_{n-1\, n-1}$ for $i\in\{1,\ldots, n-2\}$. Noticing that $f_1=a_{11}$ and $f_n=1$, we obtain
(\ref{boxcorners}).
\end{proof}

Suppose that $b>\brauer(A)$. The condition (\ref{positive_primitive_a}) (i) implies that the equation $A{\ve x}=b$ has integer solutions. Therefore, it is sufficient to show that the vector ${\ve x}$ computed by Algorithm 1 is nonnegative. When $m=1$,  (\ref{lifting}) sets  ${\ve x}=(u, w_1, \ldots, w_{n-1})^T$ with
\be\label{form_u}
u= \frac{b-a_{12}w_{1}-\cdots-a_{1 n}w_{n-1}}{a_{11}}\,.
\ee
Further, (\ref{summary}) implies that ${\ve w}=(w_1, \ldots, w_{n-1})^T\in \Lambda(A, b)$ is nonnegative and $u\in \Z$.

To see that $u\ge 0$, we observe first that the points of the affine lattice $ \Lambda(A, b)$ are split into the layers of the form
\be\label{layers}
a_{12}x_{1}+\cdots+a_{1 n}x_{n-1}=b+k a_{11}\,, k\in\Z\,.
\ee
Suppose, to derive a contradiction, that $u<0$. Then, by (\ref{form_u}),
\be \label{lower_bound_w}a_{12}w_{1}+\cdots+a_{1 n}w_{n-1}>b\,.\ee
On the other hand, by construction,  ${\ve w}\in {B}(\Lambda(A))$ and hence, using Lemma \ref{boxshape} and noticing (\ref{Brauer_bound}),
 \be  \label{upper_bound_w} a_{12}w_{1}+\cdots+a_{1 n}w_{n-1}\le G(A)+a_{11}<b+a_{11}\,.\ee
Due to (\ref{layers}), the bounds (\ref{lower_bound_w}) and (\ref{upper_bound_w}) imply ${\ve w}\notin \Lambda(A, b)$.
The obtained contradiction shows that $u\ge 0$. 

\section{Proof of Corollary \ref{polynomial_algorithm_coro}}\label{Section_coro}

We will show that a nonnegative integer solution to the system $A{\ve x}={\ve b}$ can be computed using Algorithm 1 from the proof of Theorem \ref{PolySubCone2}.
By condition (\ref{positive_primitive_a}) (i), the system $A{\ve x}={\ve b}$ is integer feasible. Following the proof of Theorem \ref{PolySubCone2}, it is sufficient to show that any ${\ve b}$ that satisfies (\ref{desired_cone_1}) must satisfy (\ref{Gomory_cone}).

 Let $h$ denote the distance from the vector ${\ve v}$ to the boundary of $\C_B$. Observe that we can write ${\ve v}={\ve v}_1+ {\ve v}_2 +{\ve p}$, where ${\ve v}_1$, ${\ve v}_2$ are  the columns of $B$ and  ${\ve p}\in \C_B$. Therefore, we have
\bea
h\ge \frac{|\det(B)|}{l_B}\,
\eea
and, consequently, the points of the affine cone
\bea
\frac{l_Bl_N}{|\det(B)|}\left(|\det(B)|-1\right){\ve v}+\C_A
\eea
are at the distance $\ge l_N(|\det(B)|-1)$ to the boundary of $\C_B$.

\section{Acknowledgement}

The author is grateful to Valentin Brimkov, Martin Henk and Timm Oertel for valuable comments and suggestions.

\end{document}